\documentclass[11pt, a4paper, oneside]{amsart}

\usepackage[english]{babel}
\usepackage{amsmath, amsthm, amsfonts, mathrsfs, amssymb}
\usepackage{mathtools}
\mathtoolsset{centercolon}
\usepackage{booktabs}
\usepackage[shortlabels]{enumitem}
\setlist[itemize]{leftmargin=20pt}
\usepackage[colorlinks, citecolor = blue]{hyperref}
\usepackage[hmargin=3cm,vmargin=3cm]{geometry}
\usepackage[color=green!40]{todonotes}
\usepackage{comment}

\usepackage{color}
\usepackage{graphicx}
\usepackage{tikz-cd}
\usetikzlibrary{arrows}

\newcommand{\N}{\ensuremath{\mathbf{N}}}
\newcommand{\Z}{\ensuremath{\mathbf{Z}}}
\newcommand{\Q}{\ensuremath{\mathbf{Q}}}
\newcommand{\R}{\ensuremath{\mathbf{R}}}

\newcommand{\mc}{\mathcal}

\DeclareMathOperator{\ind}{\mathbf{1}}

\DeclareMathOperator*{\esssup}{ess\,sup}

\renewcommand{\emptyset}{\varnothing}
\def\avint_#1{\mathchoice{\mathop{\kern 0.2em\vrule width 0.6em height 0.69678ex depth -0.58065ex \kern -0.8em \intop}\nolimits_{\kern -0.4em#1}}{\mathop{\kern 0.1em\vrule width 0.5em height 0.69678ex depth -0.60387ex \kern -0.6em \intop}\nolimits_{#1}} {\mathop{\kern 0.1em\vrule width 0.5em height 0.69678ex depth -0.60387ex \kern -0.6em \intop}\nolimits_{#1}} {\mathop{\kern 0.1em\vrule width 0.5em height 0.69678ex depth -0.60387ex \kern -0.6em \intop}\nolimits_{#1}}}

\newtheorem{theorem}{Theorem}

\newtheorem{lemma}[theorem]{Lemma}
\newtheorem{proposition}[theorem]{Proposition}

\newtheorem{TheoremLetter}{Theorem}
{}

\newtheorem{CorollaryLetter}[TheoremLetter]{Corollary}

\theoremstyle{remark}

\theoremstyle{definition}

\numberwithin{theorem}{section}
\numberwithin{equation}{section}
\title{Endpoint weak-type bounds beyond Calder\'on-Zygmund theory}

\author{Zoe Nieraeth}
\address{Zoe Nieraeth (she/her), BCAM\textendash  Basque Center for Applied Mathematics, Bilbao, Spain}
\email{zoe.nieraeth@gmail.com}
\thanks{Z. N. is supported by the grant Juan de la Cierva formación 2021 FJC2021-046837-I, the Basque Government through the BERC 2022-2025 program, by the Spanish State Research Agency project PID2020-113156GB-I00/AEI/10.13039/501100011033 and through BCAM Severo Ochoa excellence accreditation SEV-2023-2026.}

\author{Cody B. Stockdale}
\address{Cody B. Stockdale (he/him), School of Mathematical Sciences and Statistics, Clemson University, Clemson, SC 29634, USA}
\email{cbstock@clemson.edu}

\allowdisplaybreaks

\begin{document}

\begin{abstract}
We prove weighted weak-type $(r,r)$ estimates for operators satisfying $(r,s)$ limited-range sparse domination of $\ell^q$-type. Our results contain improvements for operators satisfying limited-range and square function sparse domination. In the case of operators $T$ satisfying standard sparse form domination such as Calder\'on-Zygmund operators, we provide a new and simple proof of the sharp bound 
$$
	\|T\|_{L^1_w(\R^d)\rightarrow L^{1,\infty}_w(\R^d)} \lesssim [w]_1(1+\log [w]_{\text{FW}}).
$$
\end{abstract}




\maketitle

\section{Introduction}

For a Calder\'on-Zygmund operator $T$ and $w \in A_1$, the sharp bound 
\begin{equation}\label{eq:sharpa1boundintro}
\|Tf\|_{L^{1,\infty}_w(\R^d)}\lesssim [w]_1(1+\log[w]_{\text{FW}})\|f\|_{L^1_w(\R^d)}
\end{equation}
for all $f \in L^1_w(\R^d)$ was first obtained by Lerner, Ombrosi, and P\'erez in \cite{LOP08} with the Fujii-Wilson constant $[w]_{\text{FW}}$ replaced by the larger constant $[w]_1$. This was later improved to \eqref{eq:sharpa1boundintro} by Hyt\"onen and P\'erez in \cite{HP13} and shown to be optimal with respect to $[w]_1$ in general by Lerner, Nazarov, and Ombrosi in \cite{LNO17}.

Following Hyt\"onen's resolution of the $A_2$ conjecture in \cite{Hy12}, Lerner's alternate solution by sparse domination of \cite{Le13a} revolutionized the study of weighted inequalities. In \cite{DLR16}, Domingo-Salazar, Lacey, and Rey extended \eqref{eq:sharpa1boundintro} to operators satisfying pointwise sparse domination, and in \cite{FN19}, Frey and the first author established \eqref{eq:sharpa1boundintro} for the more general class of operators $T$ satisfying sparse form domination: for all $f \in L_0^{\infty}(\R^d)$ and $g \in L^1(\R^d)$ there exists a sparse collection $\mc{S}$ for which
\[
	\int_{\R^d}|Tf||g|\,dx\lesssim \sum_{Q \in \mathcal{S}}\langle f\rangle_{1,Q}\langle g\rangle_{1,Q}|Q|.
\] 
\begin{TheoremLetter}\label{thm:A}
If $T$ satisfies sparse form domination and $w \in A_1$, then
\[
    \|Tf\|_{L^{1,\infty}_w(\R^d)}\lesssim [w]_1(1+\log [w]_{\text{FW}})\|f\|_{L^1_w(\R^d)}
\]
for all $f \in L^1(\R^d)$.
\end{TheoremLetter}
\noindent See \cite{Le19, CRr20,R21, LLS21, Ni24} for more recent related work. 

Our first main result is a new and simple proof of Theorem~\ref{thm:A}. Unlike other proofs of this bound, our technique avoids both the Calder\'on-Zygmund decomposition and the sharp reverse H\"older property for weights satisfying the Fujii-Wilson condition of \cite{HP13} -- we instead only rely on elementary dyadic methods.

Theorem~\ref{thm:A} is contained in a more general result of \cite{FN19} for operators satisfying the limited-range sparse domination introduced by Bernicot, Frey, and Petermichl in \cite{BFP16}, which applies to non-integral singular operators outside of Calder\'on-Zygmund theory. Given $0< r<s\leq \infty$, we say that $T$ satisfies $(r,s)$ limited-range sparse domination if for all $f \in L_0^{\infty}(\R^d)$ and $g \in L^1(\R^d)$ there exists a sparse collection $\mc{S}$ for which
\[
	\int_{\R^d}|Tf||g|\,dx\lesssim \sum_{Q \in \mathcal{S}}\langle f\rangle_{r,Q}\langle g\rangle_{s',Q}|Q|.
\] 
Such operators $T$ are bounded on $L^p(\R^d)$ for $p \in (r,s)$; moreover, their boundedness on the weighted Lebesgue space $L^p_w(\R^d)$ for $p \in (r,s)$ is dictated by the condition $w \in A_{p,(r,s)}$:
\[
[w]_{p,(r,s)}:=\sup_Q\langle w\rangle_{\frac{1}{\frac{1}{p}-\frac{1}{s}},Q}\langle w^{-1}\rangle_{\frac{1}{\frac{1}{r}-\frac{1}{p}},Q}<\infty,
\]
and the weak-type bound from $L_w^r(\R^d)$ to $L^{r,\infty}_w(\R^d)$ holds for $w\in A_{r,(r,s)}$. When $r=1$ and $s=\infty$, this recovers the classical Muckenhoupt condition and we simply write 
$$
    A_p:=A_{p,(1,\infty)} \quad\text{and}\quad [w]_p:=[w]_{p,(1,\infty)}.
$$
Note that $w\in A_{p,(r,s)}$ is equivalent to the full-range condition $w_{r,s}\in A_{p_{r,s}}$, where
\[
w_{r,s}:=w^{\frac{1}{\frac{1}{r}-\frac{1}{s}}},\quad p_{r,s} := \frac{\frac{1}{r}-\frac{1}{s}}{\frac{1}{p}-\frac{1}{s}},
\quad\text{and}\quad [w]_{p,(r,s)}=[w_{r,s}]^{\frac{1}{r}-\frac{1}{s}}_{p_{r,s}}.
\]

While \cite[Theorem 1.4]{FN19} indeed recovers Theorem~\ref{thm:A}, it only gives 
\[
    \|T\|_{L^r_w(\R^d)\rightarrow L^{r,\infty}_w(\R^d)}\lesssim[w^r]_1^{\frac{1}{r}}[w^r]_{\text{FW}}^{\frac{1}{r'}}(1+\log[w^r]_{\text{FW}})^{\frac{2}{r}}
\]
for $r>1$ and $T$ satisfying $(r,\infty)$ limited-range sparse domination. Our main result below improves this bound by removing the square in the logarithmic factor: if $r \ge 1$, $T$ satisfies $(r,\infty)$ limited-range sparse domination, and $w^r \in A_1$, then 
\begin{align}\label{eq:limitedrangeimprovement}
	\|T\|_{L^r_w(\R^d)\to L^{r,\infty}_w(\R^d)}\lesssim [w^r]^{\frac{1}{r}}_1[w^r]^{\frac{1}{r'}}_{\text{FW}}(1+\log[w^r]_{\text{FW}})^{\frac{1}{r}}.
\end{align}
We note that the quantitative dependence of \eqref{eq:limitedrangeimprovement} was recently obtained in the case $r=2$ for operators $T$ satisfying the stronger pointwise sparse bound  
$$
    |Tf(x)| \lesssim \sum_{Q \in \mathcal{S}} \langle f\rangle_{2,Q}\ind_Q(x) =:A_{2,\mathcal{S}}f(x)
$$
by Di Plinio, Fl\'orez-Amatriain, Parissis, and Roncal in \cite{DFPR23}. 

We also treat the following even larger class of operators: given $0<r<s\leq \infty$ and $0<q<s$, we say that an operator $T$ satisfies $(r,s)$ limited-range sparse domination of $\ell^q$-type if for all $f \in L_0^{\infty}(\R^d)$ and $g \in L^1(\R^d)$ there exists a sparse collection $\mc{S}$ for which 
\begin{align*}
\Big(\int_{\R^d}\!|Tf|^q|g|\,\mathrm{d}x\Big)^{\frac{1}{q}}\lesssim\Big(\sum_{Q\in\mc{S}}\langle f\rangle^q_{r,Q}\langle g\rangle_{(\frac{s}{q})',Q}|Q|\Big)^{\frac{1}{q}}.
\end{align*}
It was shown by Bailey, Brocchi, and Reguera in \cite{BBR23} that a large collection of non-integral square functions satisfy $(r,s)$ limited-range sparse domination of $\ell^2$-type. While \cite{BBR23} includes sharp strong-type $L^p_w(\R^d)$ bounds for such operators, they did not address the weak-type bound for $p=r$.

We note that $(r,\infty)$ limited-range sparse domination of $\ell^q$-type is implied by the bound
\begin{equation*}\label{eq:pointwisesquareintro}
|Tf(x)|\lesssim \Big(\sum_{Q\in\mc{S}}\langle f\rangle_{r,Q}^q\ind_Q(x)\Big)^{\frac{1}{q}}=:A_{r,\mc{S}}^qf(x).
\end{equation*}
The weak-type inequalities of operators satisfying this stronger property were studied by Lacey and Scurry, and Domingo-Salazar, Lacey, and Rey in \cite{LS12,DLR16} in the case $r=1$ and $q=2$, and by Hyt\"onen and Li in \cite{HL18} for $r=1$ and general $q$. The techniques of \cite{LS12,DLR16,HL18} show that the logarithmic correction in \eqref{eq:limitedrangeimprovement} is no longer needed in the bound from $L^1_w(\R^d)$ to $L^{1,\infty}_w(\R^d)$ for operators satisfying pointwise sparse domination by $A_{1,\mathcal{S}}^q$ when $q > 1$. Our main result extends this phenomenon to operators satisfying $(r,s)$ limited-range sparse domination of $\ell^q$-type for all $0<r<s\leq\infty$ and $r<q<s$.

When $s<\infty$, the result of \cite{FN19} is far from sharp. We believe that the reason for this is that the weak-type space $L^{r,\infty}_w(\R^d)$, which consists of all $f$ such that 
\[
\|f\|_{L^{r,\infty}_w(\R^d)}:=\sup_{\lambda>0}\|\lambda w\ind_{\{|f|>\lambda\}}\|_{L^r(\R^d)} <\infty,
\]
is no longer the natural choice. For general weights $u$ and $v$, Chebyshev's inequality yields
\begin{equation}\label{eq:chebyshevintro}
\|fu\|_{L^{r,\infty}(\R^d,v)}\leq\|fu\|_{L^r(\R^d,v)}=\|f\|_{L^r(\R^d,u^r v)},
\end{equation}
and so one might consider weights other than $u=1$ -- we will choose $u=w_{r,s}^{-\frac{1}{s}}$ and $v=w_{r,s}$. For $s=\infty$, the corresponding space is $L^{r,\infty}_w(\R^d)$, and for $s<\infty$, \eqref{eq:chebyshevintro} becomes
\[
\|fw_{r,s}^{-\frac{1}{s}}\|_{L^{r,\infty}(\R^d,w_{r,s})}\leq\|f\|_{L^r(\R^d,w^r)}=\|f\|_{L^r_w(\R^d)},
\]
which justifies the form of our following main result.
\begin{TheoremLetter}\label{thm:C}
If $0<r<s\leq\infty$, $0<q<s$, $T$ satisfies $(r,s)$ limited-range sparse domination of $\ell^q$-type, and $w\in A_{r,(r,s)}$, then
\[
	\|(Tf)w_{r,s}^{-\frac{1}{s}}\|_{L^{r,\infty}(\R^d,w_{r,s})}\lesssim\begin{cases}
[w]_{r,(r,s)}[w_{r,s}]^{\frac{1}{q}-\frac{1}{r}}_{\text{FW}}(1+\log[w_{r,s}]_{\text{FW}})^{\frac{1}{r}}\|f\|_{L^r_w(\R^d)} & \text{if $q\leq r$}\\
[w]_{r,(r,s)}\|f\|_{L^r_w(\R^d)} & \text{if $q> r$}
\end{cases}
\]
for all $f \in L^p_w(\R^d)$.
\end{TheoremLetter}
\noindent We emphasize here that the condition $w\in A_{r,(r,s)}$ is equivalent to the condition 
\[
w_{r,s}=w^{\frac{1}{\frac{1}{r}-\frac{1}{s}}}\in A_1,\quad [w]_{r,(r,s)}=[w_{r,s}]_1^{\frac{1}{r}-\frac{1}{s}}.
\]
The formulation of Theorem~\ref{thm:C} is corroborated by \cite[Theorem~1.11]{Zor19}, which can be used to show that if $T$ satisfies $(r,s)$ limited-range sparse domination of $\ell^q$-type, then 
\[
\|(Tf)w_{p,s}^{-\frac{1}{s}}\|_{L^{p,\infty}(\R^d;w_{p,s})}\lesssim [w]_{p,(r,s)}[w_{p,s}]_{\text{FW}}^{\frac{1}{q}-\frac{1}{p}}\|f\|_{L^p_w(\R^d)}
\]
for all $r<p<s$ and all $f \in L^r_w(\R^d)$. Our result naturally extends this to the endpoint case $p=r$ with a logarithmic correction and provides a new bound for $q>r$.

Setting $s=\infty$, Theorem~\ref{thm:C} yields the following corollary.
\begin{CorollaryLetter}\label{cor:C}
If $r,q>0$, $T$ satisfies $(r,\infty)$ limited-range sparse domination of $\ell^q$-type, and $w^r \in A_1$, then
\[
	\|T\|_{L^r_w(\R^d)\to L^{r,\infty}_w(\R^d)}\lesssim\begin{cases}
[w^r]^{\frac{1}{r}}_1[w^r]^{\frac{1}{q}-\frac{1}{r}}_{\text{FW}}(1+\log[w^r]_{\text{FW}})^{\frac{1}{r}} & \text{if $q\leq r$}\\
[w^r]^{\frac{1}{r}}_1 & \text{if $q> r$}
\end{cases}.
\]
\end{CorollaryLetter}
\noindent Note that \eqref{eq:limitedrangeimprovement} is contained in Corollary~\ref{cor:C} in the case $q=1$.

We remark that the main result of \cite{HL18} gives
$$
    \|A_{1,\mathcal{S}}^q\|_{L^p_w(\R^d)\rightarrow L^{p,\infty}_w(\R^d)} \lesssim [w]_p[w^p]_{\text{FW}}^{\frac{1}{q}-\frac{1}{p}} \lesssim [w]_{p}^{\frac{p}{q}}
$$
when $p \in (1,\infty)$ and $q<p$. Since their proof relies on the testing condition of Lacey, Sawyer, and Uriarte-Tuero from \cite{LSU09}, which only holds with $p>1$, this result does not address the $p=1$ endpoint case. It would be interesting to determine if a logarithmic factor is needed to estimate $\|A_{1,\mathcal{S}}^{q}\|_{L^1_w(\R^d)\rightarrow L^{1,\infty}_w(\R^d)}$ in terms of $[w]_1$ for $q<1$, and more generally if such a logarithmic factor is necessary in Corollary~\ref{cor:C} in the case $q<r$. Notice that this question for $A_{2,\mathcal{S}}$ is intimately connected to the open problem described in \cite[Section 1.3]{DFPR23}.

The paper is organized as follows. In Section~\ref{sec:preliminaries}, we collect notation, definitions, and preliminary results. In Section~\ref{sec:proofsketch}, we discuss our strategy and give an elementary proof of Theorem~\ref{thm:A}. In Section~\ref{sec:mainsection}, we generalize the argument for Theorem~\ref{thm:A} to prove Theorem~\ref{thm:C}.


\section{Preliminaries}\label{sec:preliminaries}

Let $d \in \N$. For $A,B>0$, we write $A \lesssim B$ if there exists $C>0$ (which possibly depends on $d$, $r$, $s$, $q$, or $T$) such that $A \leq CB$ and write $A \eqsim B$ if $A \lesssim B \lesssim A$. We additionally use the following notation: 
\begin{itemize}
    \item $L^1_{\text{loc}}(\R^d)$ is the space of locally integrable functions on $\R^d$;
    \item $w$ is a weight if $w(x)>0$ for almost every $x \in \R^d$;
    \item for a weight $w$ and $A\subseteq \R^d$, we write $w(A):=\int_A\! w\,dx$ and write $|A|$ when $w \equiv 1$;
    \item a cube is a set in $\R^d$ of the form $\prod_{j=1}^d[a_j,b_j)$ with $b_j-a_j$ equal for all $j \in \{1,\ldots,d\}$; 
    \item for a collection of cubes $\mc{P}$ and a cube $Q$, we write $\text{ch}_{\mc{P}}(Q)$ to be the collection of maximal cubes in $\mc{P}$ properly contained in $Q$;
    \item for a measurable $f$, $p>0$, and a cube $Q$, we write 
    $$
    \langle f\rangle_{p,Q}:=\left(\frac{1}{|Q|)}\int_Q\!|f|^p\,\mathrm{d
}x\right)^{1/p} \quad\text{and} \quad \langle f\rangle_{\infty,Q}:=\esssup_{x\in Q}|f(x)|;
$$ 

    \item for a weight $w$ and $p >0$, we write
$$
    \|f\|_{L^p_w(\R^d)}:=\|fw\|_{L^p(\R^d)} \quad \text{and}\quad \|f\|_{L^{p,\infty}_w(\R^d)}:=\sup_{\lambda>0}\|\lambda\ind_{\{|f|>\lambda\}}\|_{L^p_w(\R^d)};
$$
    \item for a $\sigma$-finite measure space $(\Omega,\mu)$ and $p>0$, we write
    $$
        \|f\|_{L^{p,\infty}(\Omega,\mu)} := \sup_{\lambda>0}\lambda \mu(\{x \in \Omega: |f(x)|>\lambda\})^{\frac{1}{p}}<\infty;
    $$
    \item $L_{0}^{\infty}(\R^d)$ is the space of essentially bounded functions on $\R^d$ with compact support;
    \item for $p \in [1,\infty]$, the H\"older conjugate $p'$ is defined by $\tfrac{1}{p}+\tfrac{1}{p'}=1$;
    \item for a collection of cubes $\mc{P}$, $r\ge 1$, and $f \in L^1_{\text{loc}}(\R^d)$, we write 
    $$
        M_r^{\mc{P}}f:= \sup_{Q \in \mc{P}}\langle f\rangle_{r,Q}\ind_Q,
    $$
    and we omit $r$ or $\mc{P}$ when $r=1$ or $\mc{P}$ is the collection of all cubes in $\R^d$;
    \item for a collection of cubes $\mc{P}$, $r,s>0$, and $f,g \in L^1_{\text{loc}}(\R^d)$, we write
    $$
        M^{\mc{P}}_{(r,s)}(f,g):= \sup_{Q \in \mc{P}} \langle f \rangle_{r,Q}\langle g\rangle_{s,Q}\ind_Q;
    $$
    \item we write $\sup_Q$ to indicate a supremum taken over all cubes $Q$ in $\R^d$;
    \item for a weight $w$, we write $w \in A_{\text{FW}}$ if 
$$
    [w]_{\text{FW}}:=\sup_{Q}\frac{1}{w(Q)}\int_Q\!M(w\ind_Q)\,\mathrm{d}x<\infty;
$$
\item 
for a collection of cubes $\mc{P}$, $q>0$, and $a=\{a_Q\}_{Q \in \mathcal{P}}\subseteq (0,\infty)$, we write
\[
A^q_{\mc{P}}(a):=\Big(\sum_{Q\in\mc{P}}a_Q^q\ind_Q\Big)^{\frac{1}{q}},
\]
and we omit the index $q$ from the notation when $q=1$;
\item we denote by $\|T\|_{\mathcal{X}\rightarrow \mathcal{Y}}$ the smallest constant $C>0$ such that
$$
    \|Tf\|_{\mathcal{Y}}\leq C\|f\|_{\mathcal{X}}
$$
for all $f \in \mathcal{X}$.
\end{itemize}

We say $\mc{D}$ is a dyadic grid if there exists $\alpha\in\{0,\tfrac{1}{3},\tfrac{2}{3}\}^d$ for which $\mc{D}=\mc{D}^\alpha$, where
\[
\mc{D}^\alpha:=\{2^{-j}\big([0,1)^d+\alpha+k\big):j \in \Z, \,\, k\in\Z^d\}.
\]
We will use the standard ``one-third trick" to reduce to the dyadic setting, see \cite{LN15}.
\begin{lemma}\label{lem:onethirdtrick}
For every cube $Q\subseteq \R^d$, there exists $\alpha\in\{0,\tfrac{1}{3},\tfrac{2}{3}\}^d$ and $\widetilde{Q}\in\mc{D}^\alpha$ such that $Q\subseteq\widetilde{Q}$ and $|\widetilde{Q}|\leq 6^d|Q|$
\end{lemma}
For $\eta \in (0,1)$, we say that a collection of cubes $\mc{S}$ is $\eta$-sparse if $\mc{S}=\bigcup_{\alpha \in \{0,\tfrac{1}{3},\tfrac{2}{3}\}}\mc{S}^{\alpha}$, where $\mc{S}^{\alpha}\subseteq \mc{D}^{\alpha}$ and 
$$
    \sum_{Q' \in \text{ch}_{\mc{S}^{\alpha}}(Q)}|Q'|\leq (1-\eta)|Q|
$$
for all $Q \in \mc{S}^{\alpha}$. We simply say that $\mc{S}$ is sparse when $\eta=\tfrac{1}{2}$. Note that an $\eta$-sparse collection $\mc{S}$ satisfies the following ``almost-disjointness" property: for each $Q \in \mathcal{S}$ there exists $E_Q\subseteq Q$ such that $|E_Q|\geq \eta |Q|$ and $\{E_Q\}_{Q \in \mathcal{S}}$ is a pairwise disjoint collection.

We rely on Kolmogorov's lemma, which in a $\sigma$-finite measure space $(\Omega,\mu)$, states that for all $f \in L^{p,\infty}(\Omega,\mu)$, all $E\subseteq\Omega$ with $0<\mu(E)<\infty$, and all $0<\theta<  p$, we have
\begin{equation}\label{eq:kolmogorov}
\int_E\!|f|^\theta\,\mathrm{d}\mu\leq(\tfrac{p}{\theta})'\|f\|^\theta_{L^{p,\infty}(\Omega,\mu)}\mu(E)^{1-\frac{\theta}{p}}.
\end{equation}
We also use the following variant, which extends \cite[Exercise~1.1.14]{Gr14a} to general $q>0$.
\begin{lemma}\label{lem:kolmogorovtao}
Let $(\Omega,\mu)$ be a $\sigma$-finite measure space and $p,q>0$. Then $f\in L^{p,\infty}(\Omega,\mu)$ if and only if there exists $C>0$ such that for each $E\subseteq\Omega$ with $0<\mu(E)<\infty$ there exists $E'\subseteq E$ with $\mu(E')\geq\tfrac{1}{2}\mu(E)$ and
\[
\int_{E'}\!|f|^q\,\mathrm{d}\mu\leq C\mu(E)^{1-\frac{q}{p}}.
\]
Moreover, the optimal constant $C$ satisfies
\[
C^{\frac{1}{q}} \eqsim \|f\|_{L^{p,\infty}(\Omega,\mu)}.
\]
\end{lemma}
\begin{proof}
First suppose that $f\in L^{p,\infty}(\Omega,\mu)$. Given $E\subseteq\Omega$ with $0<\mu(E)<\infty$, define
\[
\gamma:=\Big(\frac{2}{\mu(E)}\Big)^{\frac{1}{p}}\|f\|_{L^{p,\infty}(\Omega,\mu)} \quad\text{and}\quad E':=E\setminus\{x\in\Omega:|f(x)|>\gamma\}.
\]
Then, by definition
\begin{align*}
\mu(E')&\geq\mu(E)-\mu(\{x\in\Omega:|f(x)|>\gamma\})\\
&\geq \mu(E)-\tfrac{1}{\gamma^p}\|f\|^p_{L^{p,\infty}(\Omega,\mu)}\\
&=\tfrac{1}{2}\mu(E),
\end{align*}
and 
\[
\int_{E'}\!|f|^q\,\mathrm{d}\mu\leq\gamma^q\mu(E')\leq2^{\frac{q}{p}}\mu(E)^{1-\frac{q}{p}}\|f\|^q_{L^{p,\infty}(\Omega,\mu)}.
\]

Conversely, let $\lambda>0$ and choose $E=\{x\in\Omega:|f(x)|>\lambda\}$. Then, by assumption, there exists $E'\subseteq E$ with $\mu(E')\geq\tfrac{1}{2}\mu(E)$ and
\[
\lambda^q\mu(E)\leq2\lambda^q\mu(E')\leq2\int_{E'}\!|f|^q\,\mathrm{d}\mu\leq 2C\mu(E)^{1-\frac{q}{p}},
\]
and so
\[
\lambda \mu(E)^{\frac{1}{p}}\leq 2^{\frac{1}{q}}C^{\frac{1}{q}}.
\]
Taking the supremum over all $\lambda>0$ proves the assertion.
\end{proof}

Our arguments crucially rely on the following limited range analogue of the disjointness condition of \cite{LS12}, which is a standard result when $\theta=0$.
\begin{lemma}\label{lem:magic}
If $\mathcal{D}$ is a dyadic grid, $0\leq \theta <1$, $\mathcal{S}\subseteq \mathcal{D}$ satisfies
\begin{equation}\label{eq:magiclemmafinites}
\sum_{Q'\in\text{ch}_{\mc{S}}(Q)}|Q'|\leq\big(\tfrac{1-\theta}{4}\big)^{\frac{1}{1-\theta}}|Q|
\end{equation}
for all $Q \in \mathcal{S}$, $f \in L^1_{\text{loc}}(\R^d)$, $w$ is a weight, and $\lambda>0$, then for each $Q$ in 
\[
    \mathcal{E}:= \{Q\in\mc{S}:\lambda\langle w\rangle_{1,Q}^\theta<\langle f\rangle_{1,Q}\leq 2\lambda\langle w\rangle_{1,Q}^\theta\}
\]
there exists $E_Q\subseteq Q$ such that $\{E_Q\}_{Q \in \mathcal{E}}$ is a pairwise disjoint collection and 
\[
    \int_Q |f|\,dx \leq 2\int_{E_Q}|f|\,dx.
\]
\end{lemma}
\begin{proof}
Define $E_Q:=Q\setminus\bigcup_{Q'\in\text{ch}_{\mc{E}}(Q)}Q'$. Using Kolmogorov's lemma \eqref{eq:kolmogorov}, we have
\begin{align*}
\sum_{Q'\in\text{ch}_{\mc{E}}(Q)}\int_{Q'}\!|f|\,\mathrm{d}x
&\leq 2\lambda\sum_{Q'\in\text{ch}_{\mc{E}}(Q)}\langle w\rangle_{1,Q'}^{\theta}|Q'|\\
&\leq 2\lambda\int_{\bigcup_{Q'\in\text{ch}_{\mc{E}}(Q)}Q'}\!(M^{\mc{D}}(w\ind_Q))^{\theta}\,\mathrm{d}x\\
&\leq 2\lambda \Big(\frac{1}{\theta}\Big)'\|M^{\mc{D}}(w\ind_Q)\|_{L^{1,\infty}(\R^d)}\Big|\bigcup_{Q' \in \text{ch}_{\mc{E}}(Q)}Q'\Big|\\
&\leq \frac{2\lambda}{1-\theta} w(Q)^{\theta}\Big(\sum_{Q'\in\text{ch}_{\mc{E}}(Q)}|Q'|\Big)^{1-\theta}\\
&\leq\frac{1}{2}\int_Q\!|f|\,\mathrm{d}x,
\end{align*}
so that
\begin{align*}
\int_{E_Q}\!|f|\,\mathrm{d}x
&=\int_Q\!|f|\,\mathrm{d}x-\sum_{Q'\in\text{ch}_{\mc{E}}(Q)}\int_{Q'}\!|f|\,\mathrm{d}x
\geq \frac{1}{2}\int_Q\!|f|\,\mathrm{d}x,
\end{align*}
as desired.
\end{proof}

We prove our estimates using the following application of Lemma~\ref{lem:onethirdtrick} and Lemma~\ref{lem:kolmogorovtao}. 
\begin{proposition}\label{lem:formtosparsereduction}
Let $0<r<s\leq\infty$, $0<q<s$, $T$ satisfy $(r,s)$ limited-range sparse domination of $\ell^q$-type, $v:=w^{\frac{1}{\frac{1}{r}-\frac{1}{s}}}$ a weight, and $0<\eta<1$. If there is a constant $C>0$ such that for every dyadic grid $\mc{D}$ and every $E\subseteq\R^d$ with $0<v(E)<\infty$ there exists a subset $E'\subseteq E$ with
$
v(E')\geq (1-\tfrac{1}{2\cdot 3^d})v(E)
$
such that 
\[
\sum_{Q\in\mc{S}}\langle f\rangle_{r,Q}^q\langle v\ind_{E'}\rangle_{(\frac{s}{q})',Q}|Q|\leq C v(E)^{1-\frac{q}{r}}
\]
for all finite $\eta$-sparse collections $\mc{S}\subseteq\mc{D}$ and all $f\in L_{w}^{r}(\R^d)\cap L^\infty_0(\R^d)$ with $\|f\|_{L_w^r(\R^d)}=1$, then $T$ satisfies
\[
\|(Tf)v^{-\frac{1}{s}}\|_{L^{r,\infty}(\R^d;v)}\lesssim 
C^{\frac{1}{q}}\|f\|_{L^r_w(\R^d)}.
\]
for all $f\in L^r_w(\R^d)$.
\end{proposition}
\begin{proof}
By Lemma~\ref{lem:kolmogorovtao}, it suffices to show that for every $E\subseteq \R^d$ with $0<v(E)<\infty$ there exists a subset $E'\subseteq E$ such that $v(E')\geq\frac{1}{2}v(E)$ and
\[
\int_{E'}\!|T f|^qv^{1-\frac{q}{s}}\,\mathrm{d}x\lesssim 
Cv(E)^{1-\frac{q}{r}}\|f\|_{L^r_w(\R^d)}^q
\]
for all $f \in L^r_w(\R^d)$. 

Let $E\subseteq \R^d$ with $0<v(E)<\infty$, let $f \in L^r_w(\R^d)$, and assume without loss of generality $f$ is bounded with compact support and $\|f\|_{L^r_{w}(\R^d)}=1$. By hypothesis, for each $\alpha \in \{0,\tfrac{1}{3},\tfrac{2}{3}\}^d$ there exists $E'_\alpha\subseteq E$ such that $v(E'_\alpha)\geq (1-\frac{1}{2\cdot 3^d}) v(E)$ and 
\[
\sum_{Q\in\mc{S}^\alpha}\langle f\rangle_{r,Q}^q\langle v\ind_{E_\alpha'}\rangle_{(\frac{s}{q})',Q}|Q|\leq Cv(E)^{1-\frac{q}{r}}
\]
for all finite $\eta$-sparse collections $\mc{S}^\alpha\subseteq\mc{D}^\alpha$. Set $E':=\bigcap_{\alpha \in \{0,\tfrac{1}{3},\tfrac{2}{3}\}^d}E'_\alpha$. Then 
\begin{align*}
v(E')&=v(E)-v\Big(\bigcup_{\alpha\in\{0,\tfrac{1}{3},\tfrac{2}{3}\}^d}E\setminus E'_\alpha\Big)
\geq v(E)-\sum_{\alpha \in \{0,\tfrac{1}{3},\tfrac{2}{3}\}^d}v(E\setminus E'_\alpha)\\
&\geq v(E)-\frac{v(E)}{2}\sum_{\alpha\in \{0,\tfrac{1}{3},\tfrac{2}{3}\}^d}\frac{1}{3^d}=\frac{1}{2}v(E).
\end{align*}
Since $f\in L^\infty_0(\R^d)$ and (possibly after a truncation argument) $v^{1-\frac{q}{s}}\ind_{E'}\in L^1(\R^d)$, there is a sparse collection $\mc{S}$ such that
\[
\int_{E'}\!|Tf|^qv^{1-\frac{q}{s}}\,\mathrm{d}x \lesssim\sum_{Q\in\mc{S}}\langle f\rangle_{r,Q}^q\langle v\ind_{E'}\rangle^{1-\frac{q}{s}}_{1,Q}|Q|,
\]
and, by monotone convergence, it suffices to estimate the right-hand side for finite $\mc{S}$. Moreover, by Lemma~\ref{lem:onethirdtrick} and \cite[Lemma~3.2.4,\, Proposition~3.2.10]{Ni20}, there are finite $6^{-d}\eta$-sparse collections $\mc{S}^\alpha\subseteq\mc{D}^\alpha$ and $\eta$-sparse collections $\mc{E}^\alpha\subseteq\mc{S}^\alpha$ for which
\begin{align*}
\sum_{Q\in\mc{S}}\langle f\rangle_{r,Q}^q\langle v\ind_{E'}\rangle^{1-\frac{q}{s}}_{1,Q}|Q|&\lesssim
\|M^{\mc{S}}_{(\frac{r}{q},(\frac{s}{q})')}(|f|^q,v^{1-\frac{q}{s}}\ind_{E'})\|_{L^1(\R^d)}\\
&\lesssim 
\sum_{\alpha \in \{0,\tfrac{1}{3},\tfrac{2}{3}\}^d}\|M^{\mc{S}^\alpha}_{(\frac{r}{q},(\frac{s}{q})')}(|f|^q,v^{1-\frac{q}{s}}\ind_{E'})\|_{L^1(\R^d)}\\
&\lesssim 
\sum_{\alpha \in \{0,\tfrac{1}{3},\tfrac{2}{3}\}^d}\sum_{Q\in\mc{E}^\alpha}\langle f\rangle_{r,Q}^q\langle v\ind_{E'}\rangle^{1-\frac{q}{s}}_{1,Q}|Q|.
\end{align*}
We conclude that
\begin{align*}
\sum_{Q\in\mc{S}}\langle f\rangle_{r,Q}^q\langle v\ind_{E'}\rangle^{1-\frac{q}{s}}_{1,Q}|Q|
&\lesssim 
\sum_{\alpha\{0,\tfrac{1}{3},\tfrac{2}{3}\}^d}\sum_{Q\in\mc{E}^\alpha}\langle f\rangle_{r,Q}^q\langle v\ind_{E'}\rangle^{1-\frac{q}{s}}_{1,Q}|Q|\\
&\lesssim 
Cv(E)^{1-\frac{q}{r}}.
\end{align*}
The assertion follows.
\end{proof}

We also use the following extension of the weak-type estimate for the maximal operator.
\begin{lemma}\label{lem:MWeakType}
Let $\mc{D}$ be a dyadic grid, let $0\leq \theta < 1$, and let $w\in A_1$. Then 
\[
N^{\mc{D}}f(x):=\sup_{Q\in\mc{D}}\langle f\rangle_{1,Q}\langle w\rangle_{1,Q}^{-\theta}\ind_Q(x)
\]
satisfies
\[
\|N^{\mc{D}}f\|_{L^{1,\infty}(\R^d,w)}\leq[w]_1^{1-\theta}\|f\|_{L^1(\R^d,w^{1-\theta})}
\]
for all $f\in L^1(\R^d,w^{1-\theta})$.
\end{lemma}
\begin{proof}
Let $\mc{F}\subseteq\mc{D}$ be a finite collection. By monotone convergence, we need only consider $N^{\mc{F}}$ instead of $N^{\mc{D}}$, where $N^{\mc{F}}$ is defined analogously to $N^{\mc{D}}$, but with the supremum taken over $\mc{F}$ instead of $\mc{D}$. Let $\lambda>0$ and let $\mc{P}$ denote the collection of maximal cubes in 
\[
\{Q\in\mc{F}:\langle f\rangle_{1,Q}>\lambda\langle w\rangle_{1,Q}^\theta\}.
\]
Then
\begin{align*}
\lambda w(\{N^{\mc{F}}f>\lambda\})&=\sum_{Q\in\mc{P}}\lambda w(Q)
\leq \sum_{Q\in\mc{P}}\Big(\int_Q\!|f|\,\mathrm{d}x\Big)\langle w\rangle_{1,Q}^{1-\theta}\\
&\leq[w]_1^{1-\theta}\sum_{Q\in\mc{P}}\int_Q\!|f|w^{1-\theta}\,\mathrm{d}x
\leq[w]_1^{1-\theta}\|f\|_{L^1(\R^d,w^{1-\theta})},
\end{align*}
as desired.
\end{proof}


\section{An elementary proof of Theorem~\ref{thm:A}}
\label{sec:proofsketch}
We first outline our strategy when applied to a single sparse operator. Given a dyadic grid $\mc{D}$, a sparse collection $\mc{S}\subseteq\mc{D}$, and $w\in A_1$, the goal is to show that 
\begin{equation}\label{eq:sketchgoal}
\|A_{\mc{S}}f\|_{L^{1,\infty}_w(\R^d)}\lesssim [w]_1(1+\log[w]_{\text{FW}})\|f\|_{L^1_w(\R^d)}
\end{equation}
for all $f\in L^1_w(\R^d)$. We have the following consequence of the good-$\lambda$ inequality of \cite{DFPR23} from \cite[Theorem 4.2]{NSS24}.
\begin{proposition}\label{prop:sketch1}
If $\mc{D}$ is a dyadic grid, $w\in A_{\text{FW}}$, $\eta\in(0,1)$, $\mc{S}\subseteq\mc{D}$ is an $\eta$-sparse collection, $a=\{a_Q\}_{Q\in\mc{D}} \subseteq (0,\infty)$, and $t_1\leq t_2$, then 
\[
\|A^{t_1}_{\mc{S}}(a)\|_{L^{1,\infty}(\R^d,w)}\lesssim\big(\tfrac{1}{\eta}[w]_{\text{FW}}\big)^{\frac{1}{t_1}-\frac{1}{t_2}}\|A^{t_2}_{\mc{S}}(a)\|_{L^{1,\infty}(\R^d,w)}.
\]
\end{proposition}

We also have the following result.
\begin{proposition}\label{prop:sketch2}
If $\mc{D}$ is a dyadic grid, $\mc{S}\subseteq\mc{D}$ is sparse, $t>1$, and $w\in A_1$, then 
\[
\|A^t_{\mc{S}}f\|_{L^{1,\infty}_w(\R^d)}\lesssim t'[w]_1\|f\|_{L^1_w(\R^d)}
\]
for all $f\in L^1_w(\R^d)$.
\end{proposition}
\noindent Proposition~\ref{prop:sketch2} follows from the argument of \cite{LS12}, where the quantity $t'$ comes from tracking the constant in the bound obtained from Lemma~\ref{lem:magic}. In particular, splitting the sparse collection into cubes where $\langle f\rangle_{1,Q}\eqsim 2^{-k}$, one finds the bound with constant
\[
\Big(\sum_{k=0}^\infty 2^{-(t-1)k}\Big)^{\frac{1}{t}}=\big(1-2^{-(t-1)}\big)^{-\frac{1}{t}}\eqsim t'.
\]

The estimate \eqref{eq:sketchgoal} follows by combining Proposition~\ref{prop:sketch1} with Proposition~\ref{prop:sketch2} to obtain
\[
\|A_{\mc{S}}f\|_{L^{1,\infty}_w(\R^d)}\lesssim t'[w]_{\text{FW}}^{\frac{1}{t'}}[w]_1\|f\|_{L^1_w(\R^d)},
\]
for arbitrary $t>1$. Setting $t'=2+\log[w]_{\text{FW}}$ yields \eqref{eq:sketchgoal}, since then $[w]_{\text{FW}}^{\frac{1}{t'}}\leq e$. 
We extend this philosophy to operators satisfying sparse form domination and obtain an elementary proof of Theorem~\ref{thm:A}. The full proof is as follows:
\begin{proof}[Proof of Theorem~\ref{thm:A}]
Let $E\subseteq\R^d$ with $0<w(E)<\infty$. By Proposition~\ref{lem:formtosparsereduction}, it suffices to show for each dyadic grid $\mc{D}$ there exists $E'\subseteq E$ such that $w(E')\geq(1-\tfrac{1}{2\cdot 3^d})w(E)$ and 
\[
\sum_{Q\in\mc{S}}\langle f\rangle_{1,Q}\langle w\ind_{E'}\rangle_{1,Q}|Q|\lesssim [w]_1(1+\log[w]_{\text{FW}})
\]
for each finite $\tfrac{3}{4}$-sparse collection $\mc{S}\subseteq\mc{D}$ and all $f\in L^1_w(\R^d)\cap L^{\infty}_0(\R^d)$ with $\|f\|_{L^1_w(\R^d)}=1$.

Given a dyadic grid $\mc{D}$, define
\[
\gamma:=\frac{2\cdot 3^d[w]_1}{w(E)} \quad\text{and}\quad E':=\{x\in E:M^{\mc{D}}f(x)\leq\gamma\}.
\]
Then, by Lemma~\ref{lem:MWeakType} with $\theta=0$,
$$
    w(\{M^{\mc{D}}f>\gamma\}) \leq \frac{\|M^{\mathcal{D}}\|_{L^1_w(\R^d)\rightarrow L^{1,\infty}_w(\R^d)}}{\gamma}\|f\|_{L^1_w(\R^d)}= \frac{w(E)}{2\cdot 3^d},
$$
so that
$$
    w(E')\geq w(E)-w(\{M^{\mc{D}}f>\gamma\})\geq (1-\tfrac{1}{2\cdot 3^d})w(E).
$$

Let $\mc{S}\subseteq\mc{D}$ be a finite $\tfrac{3}{4}$-sparse collection. Note that for any $Q\in\mc{S}$ with $\langle f\rangle_{1,Q}>\gamma$ it holds that $Q\cap E'=\emptyset$, and hence it suffices to sum over
\[
\mc{S}_+:=\{Q\in\mc{S}:\langle f\rangle_{1,Q}\leq\gamma\}.
\]
Given $0<\lambda< 1$, let
\[
\mc{S}_\lambda:=\{Q\in\mc{S}_+: w(E\cap Q)>\lambda w(Q)\}
\]
and let $\mc{S}_{\lambda}^{\ast}$ denote the maximal elements of $\mc{S}_\lambda$. Then
\begin{equation}\label{eq:thm:Bshortr=1}
\begin{split}
\sum_{Q\in\mc{S}_+}\langle f\rangle_{1,Q}\langle w\ind_{E}\rangle_{1,Q}|Q|&= \sum_{Q \in \mc{S}_+}\langle f\rangle_{1,Q}w(Q)\int_0^{\frac{w(E \cap Q)}{w(Q)}}\,d\lambda\\
&=\int_0^1\sum_{Q\in\mc{S}_\lambda}\langle f\rangle_{1,Q}w(Q)\,\mathrm{d}\lambda\\
&=\int_0^1\sum_{Q_0\in\mc{S}^\ast_\lambda}\sum_{\substack{Q\in\mc{S}_\lambda\\Q\subseteq Q_0}}\langle f\rangle_{1,Q}w(Q)\,\mathrm{d}\lambda.
\end{split}
\end{equation}

We next estimate the inner summation. Fix $Q_0 \in S_{\lambda}^{\ast}$ and $0<\lambda <1$. Using H\"older's inequality twice, we have for any $t>1$ that 
\begin{align*}
\sum_{\substack{Q\in\mc{S}_\lambda\\Q\subseteq Q_0}}\langle f\rangle_{1,Q}w(Q)&=\int_{Q_0}\!\Big(\sum_{\substack{Q\in\mc{S}_\lambda\\Q\subseteq Q_0}}\langle f\rangle_{1,Q}\ind_Q\Big)w\,\mathrm{d}x\\
&\leq\int_{Q_0}\!\Big(\sum_{\substack{Q\in\mc{S}_\lambda\\Q\subseteq Q_0}}\langle f\rangle_{1,Q}^t\ind_Q\Big)^{\frac{1}{t}}\Big(\sum_{\substack{Q\in\mc{S}_\lambda\\Q\subseteq Q_0}}\ind_Q\Big)^{\frac{1}{t'}}w\,\mathrm{d}x\\
&\leq\Big(\sum_{\substack{Q\in\mc{S}_\lambda\\Q\subseteq Q_0}}\langle f\rangle_{1,Q}^tw(Q)\Big)^{\frac{1}{t}}\Big(\sum_{\substack{Q\in\mc{S}_\lambda\\Q\subseteq Q_0}} w(Q)\Big)^{\frac{1}{t'}}.
\end{align*}
Using the sparse and Fujii-Wilson conditions, we have 
$$
    \sum_{\substack{Q\in\mc{S}_\lambda\\Q\subseteq Q_0}} w(Q)\lesssim \sum_{\substack{Q \in \mc{S}\\Q \subseteq Q_0}}\int_{E_Q}\!M(w\ind_{Q_0})\,\mathrm{d}x \leq \int_{Q_0}\!M(w\ind_{Q_0})\,dx\leq [w]_{\text{FW}}w(Q_0),
$$
and so the above quantity is controlled by a constant times  
$$
    [w]_{\text{FW}}^{\frac{1}{t'}}w(Q_0)^{\frac{1}{t'}}\Big(\sum_{\substack{Q \in \mc{S}_{\lambda}\\ Q\subseteq Q_0}}\langle f\rangle_{1,Q}^t w(Q)\Big)^{\frac{1}{t}}.
$$

Decompose $\mc{S}_\lambda=\bigcup_{k=0}^\infty\mc{E}_k$, where
\[
\mc{E}_k:=\{Q\in\mc{S}_\lambda:2^{-(k+1)}\gamma<\langle f\rangle_{1,Q}\leq 2^{-k}\gamma\}.
\] Using the definition of $\mc{E}_k$, Lemma~\ref{lem:magic} with $\theta=0$, and the $A_1$ condition, we continue estimating by 
\begin{align*}
&[w]_{\text{FW}}^{\frac{1}{t'}}w(Q_0)^{\frac{1}{t'}}\gamma^{\frac{1}{t'}}\Big(\sum_{k=0}^\infty 2^{-(t-1)k}\sum_{\substack{Q\in\mc{E}_k\\Q\subseteq Q_0}}\langle f\rangle_{1,Q}w(Q)\Big)^{\frac{1}{t}}\\
&\quad\lesssim [w]^{\frac{1}{t'}}_{\text{FW}}w(Q_0)^{\frac{1}{t'}}[w]_1^{\frac{1}{t'}}w(E)^{-\frac{1}{t'}}\Big(\sum_{k=0}^\infty 2^{-(t-1)k}\sum_{\substack{Q\in\mc{E}_k\\Q\subseteq Q_0}}\int_{E_Q}\!|f|Mw\,\mathrm{d}x\Big)^{\frac{1}{t}}\\
&\quad\leq[w]^{\frac{1}{t'}}_{\text{FW}}w(Q_0)^{\frac{1}{t'}}[w]_1w(E)^{-\frac{1}{t'}}\big(1-2^{-(t-1)}\big)^{-\frac{1}{t}}\|f\|^{\frac{1}{t}}_{L^1_{w}(Q_0)}.
\end{align*}

Using $(1-2^{-(t-1)})^{-\frac{1}{t}}\eqsim t'$, H\"older's inequality, and $w(Q_0)\leq \lambda^{-1}w(E\cap Q_0)$, it follows from \eqref{eq:thm:Bshortr=1} that
\begin{align*}
\sum_{Q\in\mc{S}}\langle f\rangle_{1,Q}&\langle w\ind_{E'}\rangle_{1,Q}|Q|\lesssim
t'[w]^{\frac{1}{t'}}_{\text{FW}}[w]_1w(E)^{-\frac{1}{t'}}\int_0^1\!\sum_{Q_0\in\mc{S}^\ast_\lambda}\|f\|_{L^1_{w}(Q_0)}^{\frac{1}{t}}w(Q_0)^{\frac{1}{t'}}\,\mathrm{d}\lambda\\
&\leq
t'[w]^{\frac{1}{t'}}_{\text{FW}}[w]_1w(E)^{-\frac{1}{t'}}\int_0^1\!\Big(\sum_{Q_0\in\mc{S}^\ast_\lambda}\|f\|_{L^1_{w}(Q_0)}\Big)^{\frac{1}{t}}\Big(\sum_{Q_0\in\mc{S}_\lambda^\ast} w(Q_0)\Big)^{\frac{1}{t'}}\,\mathrm{d}\lambda\\
&\leq t'[w]^{\frac{1}{t'}}_{\text{FW}}[w]_1w(E)^{-\frac{1}{t'}}\int_0^1\!\lambda^{-\frac{1}{t'}}\Big(\sum_{Q_0\in\mc{S}_\lambda^\ast} w(E\cap Q_0)\Big)^{\frac{1}{t'}}\,\mathrm{d}\lambda\\
&\leq t t'[w]^{\frac{1}{t'}}_{\text{FW}}[w]_1.
\end{align*}
Setting $t'=2+\log[w]_{\text{FW}}$, we have $t\leq 2$ and $[w]^{\frac{1}{t'}}_{\text{FW}}\leq e$, so the result follows.
\end{proof}

\section{Proof of Theorem~\ref{thm:C}}\label{sec:mainsection}

\begin{proof}[Proof of Theorem~\ref{thm:C}]
Let $v:=w_{r,s}=w^{\frac{1}{\frac{1}{r}-\frac{1}{s}}}$ and let $E\subseteq\R^d$ with $0<v(E)<\infty$. By Proposition~\ref{lem:formtosparsereduction}, it suffices to show that for each dyadic grid $\mc{D}$ there exists $E'\subseteq E$ such that $v(E')\ge (1-\frac{1}{2\cdot 3^d})v(E)$ and 
\[
\sum_{Q\in\mc{S}}\langle f\rangle_{r,Q}^q\langle v\ind_{E'}\rangle^{1-\frac{q}{s}}_{1,Q}|Q|\lesssim v(E)^{1-\frac{q}{r}}\begin{cases} 
[v]_1^{\frac{q}{r}(1-\frac{r}{s})}[v]_{\text{FW}}^{1-\frac{q}{r}}(1+\log[v]_{\text{FW}})^{\frac{q}{r}} & \text{if $q\leq r$}\\
[v]_1^{\frac{q}{r}(1-\frac{r}{s})} & \text{if $q>r$}
\end{cases}
\]
for each finite $\eta_{r,s}$-sparse collection $\mc{S} \subseteq \mc{D}$, where $(1-\eta_{r,s})^{1-\frac{r}{s}}:=\tfrac{1}{4}(1-\frac{r}{s})$, and all $f \in L^r_w(\R^d)\cap L_0^{\infty}(\R^d)$ with $\|f\|_{L^r_w(\R^d)}=1$.

Given a dyadic grid $\mc{D}$, define $a_Q:=\langle |f|^r\rangle_{1,Q}\langle v\rangle_{1,Q}^{-\frac{r}{s}}$ and
\[
N^{\mc{D}}f:=\sup_{Q\in\mc{D}}a_Q\ind_Q.
\]
Setting
\[
\gamma:=\frac{2\cdot 3^d}{v(E)}[v]_1^{1-\frac{r}{s}} \quad\text{and}\quad E':=\{x\in E:N^{\mc{D}}f(x)\leq\gamma\},
\]
it follows from Lemma~\ref{lem:MWeakType} that
\[
v(E')\geq v(E)-v(\{N^{\mc{D}} f>\gamma\})\geq(1-\tfrac{1}{2\cdot 3^d})v(E).
\]
Let $\mc{S}\subseteq\mc{D}$ be a finite $\eta_{r,s}$-sparse collection. Note that for any $Q\in\mc{S}$ with $a_Q>\gamma$ we have $Q\cap E'=\emptyset$, and hence it suffices to sum over
\[
\mc{S}_+:=\{Q\in\mc{S}:\langle |f|^r\rangle_{1,Q}\leq\gamma\langle v\rangle_{1,Q}^{\frac{r}{s}}\}.
\]
We consider the cases $q> r$ and $q\leq r$ separately.

First, assume that $q>r$. Write $\mc{S}_+=\bigcup_{k=0}^\infty\mc{S}_k$, where 
\[
\mc{S}_k:=\{Q\in\mc{S}_+:2^{-(k+1)}\gamma\langle v\rangle_{1,Q}^{\frac{r}{s}}<\langle |f|^r\rangle_{1,Q}\leq 2^{-k}\gamma \langle v\rangle_{1,Q}^{\frac{r}{s}}\}.
\]
By definition of $\mathcal{S}_k$ and Lemma~\ref{lem:magic}, we have
\begin{align*}
\sum_{Q\in\mc{S}_k}\langle f\rangle_{r,Q}^q\langle v\rangle^{1-\frac{q}{s}}_{1,Q}|Q|
&\leq 2^{-(\frac{q}{r}-1)k}\gamma^{\frac{q}{r}-1}\sum_{Q\in\mc{S}_k}\langle |f|^r\rangle_{1,Q}\langle v\rangle_{1,Q}^{1-\frac{r}{s}}|Q|\\
&\lesssim 2^{-(\frac{q}{r}-1)k}\gamma^{\frac{q}{r}-1}\sum_{Q\in\mc{S}_k} \int_{E_Q}\!|f|^r\langle v\rangle_{1,Q}^{1-\frac{r}{s}}\,\mathrm{d}x\\
&\leq 2^{-(\frac{q}{r}-1)k}\gamma^{\frac{q}{r}-1}[v]^{1-\frac{r}{s}}_1\\
&\eqsim 2^{-(\frac{q}{r}-1)k} v(E)^{1-\frac{q}{r}}[v]_1^{\frac{q}{r}(1-\frac{r}{s})}.
\end{align*}
Since $\sum_{k=0}^\infty2^{-(\frac{q}{r}-1)k}\eqsim [(\frac{q}{r})']^{\frac{q}{r}}$, this proves the assertion.

Now assume that $q\leq r$. Given $0<\lambda< 1$, let
\[
\mc{S}_\lambda:=\{Q\in\mc{S}_+: v(E\cap Q)>\lambda^{\frac{1}{1-\frac{q}{s}}} v(Q)\}
\]
and let $\mc{S}_\lambda^\ast$ denote the maximal elements of $\mc{S}_\lambda$. Then
\begin{equation}\label{eq:thm:C1}
\begin{split}
\sum_{Q\in\mc{S}_+}\langle f\rangle^q_{r,Q}\langle v\ind_{E}\rangle^{1-\frac{q}{s}}_{1,Q}|Q|&=\int_0^1\sum_{Q\in\mc{S}_\lambda}\langle f\rangle_{r,Q}^q\langle v\rangle_{1,Q}^{-\frac{q}{s}}v(Q)\,\mathrm{d}\lambda\\
&=\int_0^1\sum_{Q_0\in\mc{S}^\ast_\lambda}\sum_{\substack{Q\in\mc{S}_\lambda\\Q\subseteq Q_0}}a_Q^{\frac{q}{r}}v(Q)\,\mathrm{d}\lambda.
\end{split}
\end{equation}
We next estimate the inner summation. Fix $Q_0 \in \mathcal{S}_{\lambda}^*$ and $0<\lambda<1$. Using H\"older's inequality twice, we have for any $t>1$ that 
\begin{align*}
\sum_{\substack{Q\in\mc{S}_\lambda\\Q\subseteq Q_0}}a_Q^{\frac{q}{r}}v(Q)&=\int_{Q_0}\!\Big(\sum_{\substack{Q\in\mc{S}_\lambda\\Q\subseteq Q_0}}a_Q^{\frac{q}{r}}\ind_Q\Big)v\,\mathrm{d}x\\
&\leq\int_{Q_0}\!\Big(\sum_{\substack{Q\in\mc{S}_\lambda\\Q\subseteq Q_0}}a_Q^t\ind_Q\Big)^{\frac{q}{rt}}\Big(\sum_{\substack{Q\in\mc{S}_\lambda\\Q\subseteq Q_0}}\ind_Q\Big)^{1-\frac{q}{rt}}v\,\mathrm{d}x\\
&\leq\Big(\sum_{\substack{Q\in\mc{S}_\lambda\\Q\subseteq Q_0}}a_Q^tv(Q)\Big)^{\frac{q}{rt}}\Big(\sum_{\substack{Q\in\mc{S}_\lambda\\Q\subseteq Q_0}} v(Q)\Big)^{1-\frac{q}{rt}}\\
&\lesssim [v]_{\text{FW}}^{1-\frac{q}{rt}}v(Q_0)^{1-\frac{q}{rt}}\Big(\sum_{\substack{Q\in\mc{S}_\lambda\\Q\subseteq Q_0}}a_Q^tv(Q)\Big)^{\frac{q}{rt}}.
\end{align*}
Decompose $\mc{S}_\lambda=\bigcup_{k=0}^\infty\mc{E}_k$, where
\[
\mc{E}_k:=\{Q\in\mc{S}_\lambda:2^{-(k+1)}\gamma \langle v\rangle_{1,Q}^{\frac{r}{s}}<\langle |f|^r\rangle_{1,Q}\leq 2^{-k}\gamma \langle v\rangle_{1,Q}^{\frac{r}{s}}\}.
\]
Using the definition of $\mc{E}_k$, Lemma~\ref{lem:magic}, and the $A_1$ condition, we continue estimating by 
\begin{align*}
&[v]_{\text{FW}}^{1-\frac{q}{rt}}\gamma^{\frac{q}{t'}}v(Q_0)^{1-\frac{q}{rt}}\Big(\sum_{k=0}^\infty 2^{-(t-1)k}\sum_{\substack{Q\in\mc{E}_k\\Q\subseteq Q_0}}a_Qv(Q)\Big)^{\frac{q}{rt}}\\
&\quad\lesssim[v]^{1-\frac{q}{rt}}_{\text{FW}}[v]_1^{\frac{q}{rt'}}v(E)^{-\frac{q}{rt'}}\Big(\sum_{k=0}^\infty 2^{-(t-1)k}\sum_{\substack{Q\in\mc{E}_k\\Q\subseteq Q_0}}\int_{E_Q}\!|f|^r\langle v\rangle_{1,Q}^{1-\frac{r}{s}}\,\mathrm{d}x\Big)^{\frac{q}{rt}}v(Q_0)^{1-\frac{q}{rt}}\\
&\quad\leq[v]^{1-\frac{q}{rt}}_{\text{FW}}[v]_1^{\frac{q}{r}(1-\frac{r}{s})}v(E)^{-\frac{q}{rt'}}(1-2^{-(t-1)})^{-\frac{q}{rt}}\|f\|^{\frac{q}{t}}_{L^r_w(Q_0)}v(Q_0)^{1-\frac{q}{rt}}.
\end{align*}
Using $(1-2^{-(t-1)})^{-\frac{1}{t}}\eqsim t'$, H\"older's inequality, and $v(Q_0)\leq\lambda^{-1}v(E\cap Q_0)$, 
it follows from \eqref{eq:thm:C1} that
\begin{align*}
\sum_{Q\in\mc{S}_+}&\langle f\rangle^q_{r,Q}\langle v\ind_{E}\rangle_{1,Q}|Q|\lesssim
(t')^{\frac{q}{r}}[v]^{1-\frac{q}{rt}}_{\text{FW}}[v]^{\frac{q}{r}(1-\frac{r}{s})}_1v(E)^{-\frac{q}{rt'}}\int_0^1\!\sum_{Q_0\in\mc{S}^\ast_\lambda}\|f\|_{L^r_w(Q_0)}^{\frac{q}{t}}v(Q_0)^{1-\frac{q}{rt}}\,\mathrm{d}\lambda\\
&\leq (t')^{\frac{q}{r}}[v]^{1-\frac{q}{rt}}_{\text{FW}}[v]^{\frac{q}{r}(1-\frac{r}{s})}_1v(E)^{-\frac{q}{rt'}}\int_0^1\!\Big(\sum_{Q_0\in\mc{S}^\ast_\lambda}\int_{Q_0}\!|f|^rw^r\,\mathrm{d}x\Big)^{\frac{q}{rt}}\Big(\sum_{Q_0\in\mc{S}^\ast_\lambda}v(Q_0)\Big)^{1-\frac{q}{rt}}\,\mathrm{d}\lambda\\
&\leq (t')^{\frac{q}{r}}[v]^{1-\frac{q}{rt}}_{\text{FW}}[v]^{\frac{q}{r}(1-\frac{r}{s})}_1v(E)^{-\frac{q}{rt'}}\int_0^1\!\lambda^{-(1-\frac{q}{rt})}\Big(\sum_{Q_0\in\mc{S}^\ast_\lambda}v(E\cap Q_0)\Big)^{1-\frac{q}{rt}}\,\mathrm{d}\lambda\\
&\leq \tfrac{r}{q} t (t')^{\frac{q}{r}}[v]^{1-\frac{q}{rt}}_{\text{FW}}[v]^{\frac{q}{r}(1-\frac{r}{s})}_1v(E)^{1-\frac{q}{r}}.
\end{align*}
Setting $t'=2+\log[v]_{\text{FW}}$, we have that $t\leq 2$ and $[v]^{1-\frac{q}{rt}}_{\text{FW}}=[v]^{1-\frac{q}{r}+\frac{q}{rt'}}_{\text{FW}}\leq e[v]^{1-\frac{q}{r}}_{\text{FW}}$, so the result follows.
\end{proof}

\bibliography{bieb}

\begin{thebibliography}{DFPR23}

\bibitem[BBR23]{BBR23}
J.~Bailey, G.~Brocchi, and M.C. Reguera.
\newblock Quadratic sparse domination and weighted estimates for non-integral
  square functions.
\newblock {\em J. Geom. Anal.}, 33(1):Paper No. 20, 49, 2023.

\bibitem[BFP16]{BFP16}
F.~Bernicot, D.~Frey, and S.~Petermichl.
\newblock Sharp weighted norm estimates beyond {C}alder\'on-{Z}ygmund theory.
\newblock {\em Anal. PDE}, 9(5):1079--1113, 2016.

\bibitem[CRR20]{CRr20}
M.~Caldarelli and I.~P. Rivera-R\'ios.
\newblock A sparse approach to mixed weak type inequalities.
\newblock {\em Math. Z.}, 296(1-2):787--812, 2020.

\bibitem[DFPR23]{DFPR23}
F.~{Di Plinio}, M.~{Fl{\'o}rez-Amatriain}, I.~{Parissis}, and L.~{Roncal}.
\newblock {Pointwise localization and sharp weighted bounds for Rubio de
  Francia square functions}.
\newblock arXiv:2308.01442, 2023.

\bibitem[DLR16]{DLR16}
C.~{Domingo-Salazar}, M.~Lacey, and G.~Rey.
\newblock Borderline weak-type estimates for singular integrals and square
  functions.
\newblock {\em Bull. Lond. Math. Soc.}, 48(1):63--73, 2016.

\bibitem[FN19]{FN19}
D.~Frey and Z.~Nieraeth.
\newblock Weak and {S}trong {T}ype {$A_1-A_\infty$} {E}stimates for {S}parsely
  {D}ominated {O}perators.
\newblock {\em J. Geom. Anal.}, 29(1):247--282, 2019.

\bibitem[Gra14]{Gr14a}
L.~Grafakos.
\newblock {\em Classical {F}ourier analysis}, volume 249 of {\em Graduate Texts
  in Mathematics}.
\newblock Springer, New York, third edition, 2014.

\bibitem[HL18]{HL18}
T.P. Hyt\"onen and K.~Li.
\newblock Weak and strong {$A_p$}-{$A_\infty$} estimates for square functions
  and related operators.
\newblock {\em Proc. Amer. Math. Soc.}, 146(6):2497--2507, 2018.

\bibitem[HP13]{HP13}
T.P. Hyt\"onen and C.~P\'erez.
\newblock Sharp weighted bounds involving {$A_\infty$}.
\newblock {\em Anal. PDE}, 6(4):777--818, 2013.

\bibitem[Hyt12]{Hy12}
T.P. Hyt\"onen.
\newblock The sharp weighted bound for general {C}alder\'on-{Z}ygmund
  operators.
\newblock {\em Ann. of Math.}, 175(3):1473--1506, 2012.

\bibitem[Ler13]{Le13a}
A.K. Lerner.
\newblock A simple proof of the {$A_2$} conjecture.
\newblock {\em Int. Math. Res. Not.}, (14):3159--3170, 2013.

\bibitem[Ler19]{Le19}
A.K. Lerner.
\newblock A weak type estimate for rough singular integrals.
\newblock {\em Rev. Mat. Iberoam.}, 35(5):1583--1602, 2019.

\bibitem[LLS21]{LLS21}
J.~{Li}, C.~{Liang}, and C.~{Shen}.
\newblock A weak type estimate for regular fractional sparse operators.
\newblock arXiv:2109.01977, 2021.

\bibitem[LN18]{LN15}
A.K. Lerner and F.~Nazarov.
\newblock Intuitive dyadic calculus: The basics.
\newblock {\em Expositiones Mathematicae}, 2018.

\bibitem[LNO20]{LNO17}
A.K. Lerner, F.~Nazarov, and S.~Ombrosi.
\newblock On the sharp upper bound related to the weak {M}uckenhoupt-{W}heeden
  conjecture.
\newblock {\em Anal. PDE}, 13(6):1939--1954, 2020.

\bibitem[LOP08]{LOP08}
A.K. Lerner, S.~Ombrosi, and C.~P\'{e}rez.
\newblock Sharp {$A_1$} bounds for {C}alder\'{o}n-{Z}ygmund operators and the
  relationship with a problem of {M}uckenhoupt and {W}heeden.
\newblock {\em Int. Math. Res. Not. IMRN}, (6):Art. ID rnm161, 11, 2008.

\bibitem[LS12]{LS12}
M.T. {Lacey} and J.~{Scurry}.
\newblock Weighted weak type estimates for square functions.
\newblock arXiv:1211.4219, 2012.

\bibitem[LSU09]{LSU09}
M.~Lacey, E.T. Sawyer, and I.~{Uriarte-Tuero}.
\newblock {Two weight inequalities for discrete positive operators}.
\newblock arXiv:0911.3437, 2009.

\bibitem[Nie20]{Ni20}
Z.~Nieraeth.
\newblock {\em Sharp estimates and extrapolation for multilinear weight
  classes}.
\newblock PhD thesis, Delft University of Technology, 2020.

\bibitem[Nie24]{Ni24}
Z.~Nieraeth.
\newblock The {M}uckenhoupt condition.
\newblock arXiv.2405.20907, 2024.

\bibitem[NSS24]{NSS24}
Z.~{Nieraeth}, C.B. {Stockdale}, and B.~{Sweeting}.
\newblock {Weighted weak-type bounds for multilinear singular integrals}.
\newblock arXiv:2401.15725, 2024.

\bibitem[Rah21]{R21}
R~Rahm.
\newblock Borderline weak-type estimates for sparse bilinear forms involving
  $a_{\infty}$ maximal functions.
\newblock {\em J. Math. Anal. Appl.}, 504(1):Paper No. 125372, 10 pp., 2021.

\bibitem[{Zor}19]{Zor19}
P.~{Zorin-Kranich}.
\newblock ${A}_p - {A}_{\infty}$ estimates for multilinear maximal and sparse
  operators.
\newblock {\em J. Anal. Math.}, 138(2):871--889, 2019.

\end{thebibliography}
\bibliographystyle{alpha}
\end{document}